\documentclass[a4paper,12pt, oneside]{amsart}

\usepackage{amsthm, amsmath}
\usepackage{amssymb,amsfonts,amscd,verbatim, longtable}
\usepackage{xy}
\usepackage{mathrsfs}
\usepackage{wrapfig}
\usepackage{tikz}
\usepackage{fancyvrb}
\xyoption{all}
\SelectTips{cm}{10}
\usepackage{ifpdf}
\ifpdf
\usepackage{hyperref}
\else
\usepackage[hypertex]{hyperref}
\fi

\newtheorem{tr}{Theorem}[section]
\newtheorem*{tr*}{Theorem}
\newtheorem{lemma}[tr]{Lemma}
\newtheorem{pr}[tr]{Proposition}

\theoremstyle{remark}
\newtheorem{rem}[tr]{Remark}

\newlength{\myevenmargin}
\def\differential{d}
\renewcommand\d\differential

\DeclareMathOperator\im{Im}

\DeclareMathOperator\Aut{Aut}

\DeclareMathOperator\GL{GL}

\DeclareMathOperator\Sing{Sing}

\DeclareMathOperator\Supp{{\rm Supp}}

\def\k{\Bbbk}
\renewcommand\Im\im
\def\bb#1{\mathbb #1}
\def\cal#1{\mathcal #1}

\def\frak#1{\mathfrak{#1}}

\def\ra{\rightarrow}
\def\xra{\xrightarrow}

\def\mto{\mapsto}

\def\pmat#1{\begin{pmatrix}#1\end{pmatrix}}

\def\smat#1{\left(\begin{smallmatrix}#1\end{smallmatrix}\right)}

\def\point#1{\langle #1 \rangle}

\def\refeq#1{$(\ref{#1})$}
\def\pt{\mathrm{pt}}

\def\P{\bb P}

\def\O{\cal O}

\let\star *
\let\subset\subseteq

\def\defeq{:=}

\def\iso{\cong}

\title[On the singular sheaves in the fine Simpson moduli spaces]{On the singular sheaves in the fine Simpson moduli spaces of $1$-dimensional sheaves supported on plane quartics.}
\author{Oleksandr Iena}
\address{University of Luxembourg, Campus Kirchberg\\
Mathematics Research Unit\\
6, rue Richard Coudenhove-Kalergi\\
L-1359 Luxembourg City\\
Grand Duchy of Luxembourg}
\email{oleksandr.iena@uni.lu}
\date{}


\begin{document}
\begin{abstract}
In the case of the fine Simpson moduli spaces of $1$-dimensional sheaves supported on plane quartics, the subvariety of sheaves that are not locally free on their support is connected, singular, and has codimension $2$.
\end{abstract}
\maketitle

\section{Introduction}
It was shown by C.~Simpson in~\cite{Simpson1} that
for an arbitrary smooth projective variety $X$  and for an arbitrary
numerical  polynomial $P\in\bb Q[m]$ there is a coarse projective moduli space
$M\defeq M_P(X)$ of semi-stable sheaves on $X$ with Hilbert polynomial $P$.

\subsection*{Singular sheaves}
In general $M$ contains  a closed subvariety $M'$ of sheaves that are not locally free on their support. Its complement $M_B$ is then  an open dense subset whose points are sheaves that are locally free on their support.
So, one could consider $M$ as a compactification of $M_B$. We call the sheaves
from the boundary $M'=M\smallsetminus M_B$ \textit{singular}.
The boundary $M'$ does not have the minimal codimension in general. Loosely speaking, one glues together too many different directions at infinity.

\subsection*{First examples on a projective plane}
Let $\k$ be an algebraically closed field of characteristic zero, let $V$ be a vector space over $\k$ of dimension $3$, and let $\P_2=\P V$ be the corresponding projective plane. Let $P(m)=am+b$, $a\in \bb Z_{>0}$, $b\in \bb Z$ be a linear Hilbert polynomial.
Notice that twisting with $\O_{\P_2}(1)$ gives the isomorphism $M_{am+b}(\P_2)\iso M_{am+b+a}(\P_2)$. Moreover, by the duality result from~\cite{MaicanDuality}, $M_{am+b}(\P_2)\iso M_{am+a-b}(\P_2)$. Therefore, for fixed $a$, it is enough to understand at most $a/2+1$ different moduli spaces.

If $a=1$, then $M_{m+b}$ is a fine moduli space that consists of twisted structure sheaves $\O_L(b-1)$ of lines $L$ in $\P_2$. Therefore, each $M_{m+b}$ is just the dual projective plane $\P_2^*=\P V^*$. In this case there are no singular sheaves.

If $a=2$ and $b=2\beta +1$ is odd, then $M_{2m+b}$ is a fine moduli space whose points are the isomorphism classes of twisted structure sheaves $\O_C(\beta)$ of planar conics $C\subset \P_2$. In this case $M_{2m+b}$ is isomorphic to the space of conics $\P S^2V^*$, as in the previous case the subvariety $M'_{2m+b}$ of singular sheaves is empty.

The situation changes for $a=3$. For $b\in \bb Z$ with $\gcd(3, b)=1$ all moduli spaces $M_{3m+b}$ are isomorphic to the universal plane cubic curve and $M'_{3m+b}$ is a smooth subvariety of codimension $2$  isomorphic to the universal singular locus of a cubic curve (cf.~\cite{IenaUnivCurve}).

\subsection*{The main result of the paper}
In this note we study the subvariety of singular sheaves
 in the case of $M=M_{4m+a}(\P_2)$, $\gcd(4, a)=1$, i.~e., for the fine Simpson moduli spaces, which consist entirely of stable points and parameterize the isomorphism classes of sheaves. As already mentioned above, it is enough to consider the case $a=-1$.

 In~\cite{LePotier} it has been shown that $M$ is a smooth projective variety of dimension $17$.
 The main result of the paper is the following statement.
\begin{pr}\label{pr: main}
Let $M$ be the Simpson moduli space of semi-stable sheaves with Hilbert polynomial
$P(m)=4m+a$, $\gcd(4, a)=1$. Let $M'\subset M$ be the subvariety of singular sheaves. Then $M'$ is a singular (path-)connected subvariety of codimension $2$.
\end{pr}
\subsection*{Structure of the paper} In Section~\ref{section: description of strata} we give a detailed description of the stratification from~\cite{DrezetMaican4m} of the moduli space $M$ into an open stratum $M_0$ and its closed complement $M_1$.
In Section~\ref{section: M_0 in B} we describe the open stratum of $M$ as an open subvariety of a projective bundle over the space of Kronecker modules $N=N(3; 2, 3)$. In Section~\ref{section: singular sheaves} we give a characterization of singular sheaves in $M_0$ and study the fibres of $M_0$ over $N$, which allows us to demonstrate in Section~\ref{section: main result} the assertions of Proposition~\ref{pr: main}. In Section~\ref{section: sing and sing sheaves} we study,  for an isomorphism class $[\cal E]$ in $M_0$, how being singular is related with the singularities of the support of $\cal E$. The computations with \textsc{Singular} \cite{SingularProgram} used in the paper (the code and its output) are presented in Appendix~\ref{appendix: computations}.
\section{Description of $M_{4m-1}(\P_2)$}\label{section: description of strata}
Let $M$ be the Simpson moduli space of (semi-)stable sheaves on $\P_2$ with Hilbert polynomial $4m-1$.
In~\cite{DrezetMaican4m} it has been shown that $M$ can be decomposed into two strata $M_1$ and $M_0$ such that $M_1$ is a closed subvariety of $M$ of codimension  $2$ and $M_0$ is its open complement.

 \subsection{Closed stratum.}
 The closed stratum $M_1$ is a closed subvariety of $M$ of codimension $2$ given by the condition $h^0(\cal E)\neq 0$ (more precisely $h^0(\cal E)=1)$. It consists of the isomorphism classes of  sheaves with locally free resolutions
\begin{equation}\label{eq: res1}
0\ra 2\O_{\P_2}(-3)\xra{\smat{z_1&q_1\\z_2&q_2}} \O_{\P_2}(-2)\oplus \O_{\P_2}\ra \cal E\ra 0,
\end{equation}
where $z_1$ and $z_2$ are linear independent linear forms on $\P_2$.  $M_1$ is a geometric quotient of the variety of injective matrices $\smat{z_1&q_1\\z_2&q_2}$ as above by the non-reductive group
\[
( \Aut(2\O_{\P_2}(-3))\times \Aut(\O_{\P_2}(-2)\oplus \O_{\P_2}) )/\bb C^*
\]
(cf.~\cite{Drezet-Trm}). $M_1$ is isomorphic to the universal quartic plane curve
\[
\{(p, C) \mid \text{$C\subset \P_2$ is a quartic plane curve, $p\in C$}\}.
\]
The latter can be explained as follows. The sheaves with resolution~\refeq{eq: res1} are exactly the non-trivial extensions
\begin{equation}\label{eq: ext closed fibre}
0\ra \O_C\ra \cal E\ra \k_p\ra 0,
\end{equation}
where $C=C_A=Z(\det A)$ is the quartic curve defined by the determinant of $A$ and $p=p_A=Z(z_1,z_2)$ is the point on $C$ defined by two linear independent linear forms $z_1$ and $z_2$.

\subsection{Open stratum.}
The open stratum $M_0$ is the complement of $M_1$ given by the condition $h^0(\cal E)=0$,
it consists of the cokernels  $\cal E_A$ of the injective morphisms
\begin{equation}\label{eq: res0}
\O_{\P_2}(-3)\oplus 2\O_{\P_2}(-2)\xra{A} 3\O_{\P_2}(-1)
\end{equation}
with
\[
A=\pmat{q_0&q_1&q_2 \\z_0&z_1&z_2\\w_0&w_1&w_2}
\]
such that the $(2\times 2)$-minors of the linear part of $A$
are linear independent. Equivalently, the Kronecker module
\begin{equation}\label{eq: linear submatrix 2x3}
\alpha=\pmat{z_0&z_1&z_2\\w_0&w_1&w_2}
\end{equation}
is stable (cf.~\cite[Lemma~1]{Ellingsrud},~\cite[Proposition~15]{Drezet}).

\subsubsection{Twisted ideals of $3$ non-collinear points of $C$}
If the maximal minors of $\alpha$  are coprime, then  $\cal E_A\iso \cal I_Z(1)$, where $\cal I_Z$ is the ideal sheaf of the zero dimensional subscheme $Z\subset C$ of length $3$ defined by the maximal minors of $\alpha$. In this case the isomorphism class of $\cal E=\cal E_A$ is a part of the exact sequence
\begin{equation}\label{eq: twisted ideal}
0\ra \cal E\ra \O_C(1)\ra\O_Z\ra 0
\end{equation}
and is uniquely defined by $Z$ and $C$.

Let $M_{00}$ denote the open subscheme of all such sheaves in $M_0$.

\subsubsection{Extensions}
If the maximal minors of  $\alpha$ have a linear common factor, say  $l$, then $f=\det(A)=l\cdot h$ and $\cal E_A$ is in this case a non-split extension
\begin{equation}\label{eq: extension}
0\ra \O_L(-2)\ra \cal E_A\ra \O_{C'}\ra 0,
\end{equation}
where $L=Z(l)$, $C'=Z(h)$.

For fixed $L$ and $C'$ the subscheme of the isomorphism classes of non-trivial extensions
\refeq{eq: extension}
can be identified with $\k^2$.

Let $M_{01}$ denote the closed subscheme of $M_0$  of all such sheaves. Notice that $M_{01}$ is locally closed in $M$.

\subsubsection{$M_0$ as a geometric quotient.}
$M_0$ is the geometric quotient of the variety of injective matrices as in \refeq{eq: res0} by the group
\[
G'=\Aut(\O_{\P_2}(-3)\oplus 2\O_{\P_2}(-2)) \times\Aut(3\O_{\P_2}(-1)).
\]
As shown in~\cite{MaicanTwoSemiSt} $M_0$ can be seen as an open subvariety in the projective quotient $\bb B$ of the variety of all semistable matrices \refeq{eq: res0} by the same group.

\section{Description of $M_0$ as an open subvariety in $\bb B$}\label{section: M_0 in B}

\subsection{Kronecker modules}
Let $\bb V$ be the affine variety of Kronecker modules
\begin{equation}\label{eq:Kronecker}
2\O_{\P_2}(-1)\xra{\Phi} 3\O_{\P_2}.
\end{equation}
There is a natural group action of $G=(\GL_{2}(\k)\times\GL_3(\k))/\k^*$ on $\bb V$.
Since $\gcd(2, 3)=1$, all semistable points of this action are stable and $G$ acts freely on the open subset $\bb V^s$ of stable points. A Kronecker module~\refeq{eq:Kronecker} is stable if its maximal minors are linear independent quadratic forms.
There exists a geometric quotient $N=N(3; 2, 3)=\bb V^s/G$, which is a smooth projective variety of dimension $6$. For more details  consult~\cite[Section 6]{EllStro} and~\cite[Section~III]{Drezet}.

The cokernel of a stable Kronecker module $\Phi\in \bb V^s$ is an ideal of a zero-dimensional scheme $Z$ of length $3$ if the maximal minors of $\Phi$ are coprime. In this case  there is a locally free resolution
\begin{equation}\label{eq:sequence for points ideal}
0\ra 2\O_{\P_2}(-3)\xra{\Phi} 3\O_{\P_2}(-2)\xra{ \pmat{d_0\\d_1\\d_2 } } \O_{\P_2}\ra \O_{Z}\ra 0
\end{equation}
and, moreover, $Z$ does not lie on a line.
Let $\bb V_0$ denote the open subvariety of $\Phi\in \bb V^s$ of Kronecker modules with coprime maximal minors. Let $N_0\subset N$ be the corresponding open subvariety in the quotient space.

This way one obtains a morphism from $N_0\subset N$  to the Hilbert scheme $H$ of zero-dimensional subschemes of $\P_2$ of length $3$, which sends a class of $\Phi\in \bb V^s$ to the zero scheme of its maximal minors. Since, by Hilbert-Burch theorem, every zero dimensional scheme of length $3$ that does not lie on a line has a minimal resolution of type~\refeq{eq:sequence for points ideal}, this gives an isomorphism between $N_0$ and the open subvariety $H_0\subset H$ consisting of $Z$ that do not lie on a line.

Let $N'=N\setminus N_0$, then $N'$ is the quotient of the variety of Kronecker modules~\refeq{eq:Kronecker} whose maximal minors have a common linear factor.

Since every matrix representing a point in $N'$ is equivalent to a matrix $\smat{z_0&0&z_1\\0&z_0&z_2}$ with linear independent linear forms $z_0, z_1, z_2$, one can see that $N'$ is isomorphic to $\P_2^*=\P V^*$, the space of lines in $\P_2$,  such that a line corresponds to the common linear factor of the minors of the corresponding Kronecker module.

The complement $H'$ of $H_0$ is an irreducible hypersurface (cf.~\cite[p.~46]{DrezetAltern}, \cite{EllStro}). The isomorphism $H_0\ra N_0$ can be extended to the morphism $H\xra{\pi} N$ that describes $H$ as the blowing up of $N$ along $N'$. The fibre over $L\in \P_2^*$ consists of those $Z\in H$ lying on $L$, i.~e., the fibre over $L$ is $L^{[3]}\iso \P_3$, the Hilbert scheme of $3$ points on $L$.

\subsection{$\bb B$ as a projective bundle over $N$} Let us provide here the argument from the proof of~\cite[Proposition~7.7]{MaicanTwoSemiSt}.

Consider two vector spaces $\bb U_1=2\Gamma(\P_2, \O_{\P_2}(1))$ and
$\bb U_2=3\Gamma(\P_2, \O_{\P_2}(2))$.
One identifies elements of $\bb V\times \bb U_2$ with morphisms~\refeq{eq: res0} by
\[
(\Phi, Q)\mto \pmat{Q\\\Phi}.
\]
Both $\bb V\times \bb U_1$ and $\bb V\times \bb U_2$ are trivial vector bundles over $\bb V$.
Consider the morphism
\[
\bb V\times \bb U_1\xra{F} \bb V\times \bb U_2\iso \bb W, \quad (\Phi, L)\mto \pmat{L\cdot \Phi\\\Phi}.
\]
Since the matrices from $\bb V^s$ have linear independent maximal minors,
$F$ is injective over $\bb V^s$. Therefore, $\bb V^s\times \bb U_1\xra{F}\bb V^s\times \bb U_2$ is a vector subbundle and  hence the cokernel $E$ of $F$ is a vector bundle of rank $12$ on $\bb V^s$.

The group action of $\GL_{2}(\k)\times\GL_3(\k)$ on $\bb V^s\times \bb U_1$ and $\bb V^s\times \bb U_2$ induces a group action of $\GL_{2}(\k)\times\GL_3(\k)$ on $E$ and hence an action of  $G=(\GL_{2}(\k)\times\GL_3(\k))/\k^*$ on the projective bundle $\P E$.
Finally, since the stabilizer of $\Phi\in \bb V^s$ under the action of $G$ is trivial, $G$ acts trivially on the fibres of $\P E$ and thus $\P E$ descends to a projective $\P_{11}$-bundle
\[
\bb B\xra{\nu} N=N(3; n-1, n)=\bb V^s/G,
\]
which is exactly the geometric  quotient of $\bb V^s\times \bb U_2\setminus \Im F$ with respect to $G'$ mentioned above.

\subsubsection{The fibres of $\bb B\xra{\nu} N$ over $N_0$} A fibre over a point from  $N_0$ can be seen as the space of plane quartics through the corresponding subscheme of $3$ non-collinear points. Indeed, consider a point from $N_0$ given by a Kronecker module $\smat{z_0&z_1&z_2\\w_0&w_1&w_2}$ with coprime minors $d_0, d_1, d_2$. The fibre over such a point consists of the orbits of injective matrices
\[
\pmat{q_0&q_1&q_2 \\z_0&z_1&z_2\\w_0&w_1&w_2}, \quad q_0, q_1, q_2\in S^2V^*,
\]
under the group action of $G'$. In particular such a fibre is contained in $M_{00}$.
If two matrices
\[
\smat{q_0&q_1&q_2 \\z_0&z_1&z_2\\w_0&w_1&w_2}, \quad \smat{Q_0&Q_1&Q_2 \\z_0&z_1&z_2\\w_0&w_1&w_2}
\]
lie in the same orbit of the group action, then their determinants are equal up to a multiplication by a non-zero constant. Vice versa, if the determinants of two such matrices are equal, $q-Q=(q_0-Q_0, q_1-Q_1, q_2-Q_2)$ lies in the syzygy module of  $\smat{d_0\\d_1\\ d_2}$, which is generated by the rows of $\smat{z_0&z_1&z_2\\w_0&w_1&w_2}$ by Hilbert-Burch theorem. This implies that $q-Q$ is a combination of the rows and thus the matrices lie on the same orbit.

\subsubsection{$M_0$ and flags of subschemes on $\P_2$.}

Let $\bb B_{0}$ denote the restriction of $\bb B$ to $N_0$. Then $\bb B_0$ coincides with $M_{00}$ as the fibres over $N_0$ are contained in $M_0$.

Let $\P S^4 V^*=\P_{14}$ be the space of plane quartics. Let
\[
M\xra{\mu}\P S^4 V^*=\P_{14}, \quad [\cal E]\mto \Supp(\cal E),
\]
be the morphism sending an isomorphism class of sheaf $\cal E$ to its support. Then its restriction to $M_0$ is induced by the equivariant morphism that sends a matrix \refeq{eq: res0} defining a point in $M_0$ to the quartic determined by its determinant.

$\bb B_0$ is isomorphic to the image of the injective morphism
\begin{equation}\label{eq: map to flags}
\bb B_0\xra{\mu\times\nu} \P(S^4V^*) \times N_0\iso \P(S^4V^*) \times H_0,
\end{equation}
which coincides with the subvariety of
 pairs $(C, Z)$ with $Z\subset C$.
It is isomorphic to the open subscheme  $H_0(3, 4)\subset H(3, 4)$ of the Hilbert flag-scheme of flags  $Z\subset C\subset \P_2$ (zero-dimensional subscheme $Z$ of length $3$ on a curve $C\subset \P_2$ of degree $4$) such that $Z$ does not lie on a line.

\subsubsection{The fibres of $\bb B\xra{\nu} N$ over $N'$}\label{subsubsection: fibres over N'}
A fibre over $L\in N'$ can be seen as the join $J(L^*, \P S^3 V^*)\iso \P_{11}$ of $L^*\iso \P_1$ and the space of plane cubic curves $\P(S^3 V^*)\iso \P_9$. To see this assume $L=Z(x_0)$, i.~e., $L$ is given by $\smat{x_0&0&x_1\\0&x_0&x_2}$. Then the fibre over $L$ is given by the orbits of matrices
\begin{equation}\label{eq: fibre over N'}
\pmat{q_0(x_1, x_2)&q_1(x_1, x_2)&q_2(x_0, x_1, x_2)\\x_0&0&x_1\\0&x_0&x_2}
\end{equation}
and can be identified with the projective space
$\P(2H^0(L, \O_{L}(2))\oplus S^2V^*)$.
Rewrite the matrix~\refeq{eq: fibre over N'} as
\[
\smat{l\cdot x_2 -b(x_1, x_2)&-l\cdot x_1- c x_1^2 &a(x_0, x_1, x_2)\\x_0&0&x_1\\0&x_0&x_2}, \quad l(x_1, x_2)=\xi_1x_1+\xi_2x_2, \quad \xi_1, \xi_2\in \k.
\]
Its determinant equals
\(
x_0(a(x_0, x_1, x_2)\cdot x_0+b(x_0, x_1)\cdot x_1+ c\cdot x_2^2 ).
\)
This allows to reinterpret the fibre as the projective space
\[
\P( H^0(L, \O_{L}(1))\oplus S^3V^* )\iso J(L^*, \P S^3V^*).
\]
The intersection of the fibre with $M_0$ is $J(L^*, \P(S^3 V^*))\setminus L^*$. It is a rank $2$ vector bundle over $\P(S^3 V^*)$ whose fibre over a cubic curve $C'\in \P S^3 V^*$ is identified with the set of the isomorphism classes of sheaves from $M_{01}$ defined by~\refeq{eq: extension} with fixed $L$ and $C'$. This fibre
corresponds to the projective plane joining $C'$ with $L^*$ inside the join $J(L^*, \P(S^3 V^*))$. In the notations of the example above $\xi_1$ and $\xi_2$ are the coordinates of this affine plane.
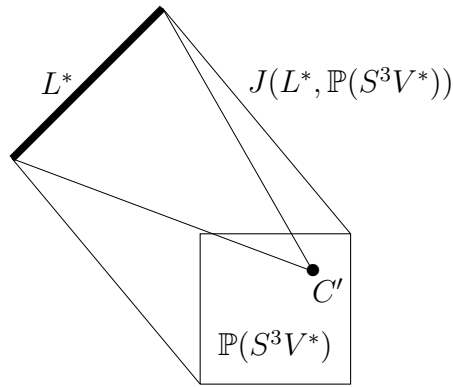
\begin{center}
\begin{tikzpicture}
\draw[line width=1mm] ++(0, 0) to ++(2, 2);
\path node at (0.6,1) {$L^*$};
\path node at (4.5,1) {$J(L^*, \P(S^3V^*))$};
\draw[very thin] ++ (0, 0) to ++(2.5cm, -3cm);
\draw[very thin] ++(2, 2) to (4.5cm, -1cm);
\draw[line width=0mm] (4, -1.5) to (0, 0) to (2, 2) to (4, -1.5);
\path node at (4, -1.5) {\textbullet};
\path node at (4.2, -1.75) {$C'$};
\begin{scope}[shift={(2.5,-3)}]
\draw ++(0, 0) to ++(2, 0) to ++(0, 2) to ++(-2, 0) to ++(0,-2);
\path node at (1,0.5) {$\P(S^3V^*)$};
\end{scope}
\end{tikzpicture}
\end{center}
The points of $J(L^*, \P(S^3 V^*))\setminus L^*$ parameterize the extensions~\refeq{eq: extension} from $M_{01}$ with fixed $L$.

\subsubsection{Description of the complement of $M_0$ in $\bb B$.}
Let $\bb B'=\bb B\setminus M_0$. Then $\bb B'$ is a union of lines $L^*$ from each fibre over $N'$ (as explained above), it is isomorphic to the tautological $\P_1$-bundle over $N'=\P_2^*$
\begin{equation}\label{eq: B'}
\{(L, x)\in \P_2^*\times \P_2 \mid L\in \P_2^*, x\in L\}.
\end{equation}
The fibre $\P_1$ of $\bb B'$ over, say, line $L=Z(x_0)\subset \P_2$ can be identified with the space of classes of matrices~\refeq{eq: res0} with zero determinant
\[
\smat{
\xi\cdot x_2&-\xi\cdot x_1&0\\
x_0&0&x_1
\\0&x_0&x_2},\quad  \xi=\alpha x_1+\beta x_2,\quad \point{\alpha, \beta}\in \P_1.
\]

Let $N_c$ be the open subset of $N_0$ that corresponds to $3$ different (and hence non-collinear) points.  Under the isomorphism $N_0\iso H_0$ it corresponds to the open  subvariety $H_c\subset H_0$ of the non-collinear configurations of $3$ points on $\P_2$.

Let $M_c=\bb B_c$ be the restriction of $\bb B$ to $N_c$. Then $M_{c}\subset M_{00}\subset M_0\subset M$ are inclusions of open subvarieties of  $M$.

\section{The subvariety of singular sheaves}\label{section: singular sheaves}
Let $M'_1$ and $M'_0$ denote the intersections of the subvariety $M'=M'_{4m-1}$ of singular sheaves with $M_1$ and $M_0$ respectively.
\subsection{Characterization of singular sheaves}
\subsubsection{Singular sheaves in $M_1$}\label{subsubsection: singular in M_1}
As shown in ~\cite{IenaUnivCurve}, the subvariety $M'_1$ coincides  with the universal singular locus
\[
\{(p, C) \mid \text{$C\subset \P_2$ is a quartic plane curve, $p\in \Sing(C)$}\},
\]
which is a smooth subvariety of $M_1$ of codimension $2$.

\subsubsection{Singular sheaves in $M_0$.}
\begin{lemma}\label{lemma: singularMin}
The sheaf $\cal E_A$ from $M_0$ is singular if and only if the ideal $\cal I_{min}=\cal I_{min}(A)$ generated by all $(2\times 2)$-minors of $A$ defines a non-empty scheme.
\end{lemma}
\begin{proof}
If there are no zeros of $\cal I_{min}$, then at every point of $\P_2$ at least one of the $(2\times 2)$-minors is invertible, hence using invertible elementary transformations one can bring $A$ to the form
\[
\pmat{1&0&0\\0&1&0\\0&0&\det A}
\]
and therefore $\cal E$ is locally isomorphic to $\O_C$, $C=Z(\det A)=\Supp\cal E$.

If $p$ is a zero point of $\cal I_{min}$,
then the rank of $A$ is at most $1$ at $p$. Therefore, the dimension of $\cal E(p)=\cal E_p/\frak{m}_p\cal E_p $ is at least $2$. Since the rank of $\cal E$ (on support) is $1$, we conclude that $\cal E$ is a singular sheaf.
\end{proof}

\subsection{Fibres of $M'_0$ over $N$}\label{subsection: fibres of M'}
Let us consider the restriction of $\nu$ to $M'_0$ and describe its fibres. There are the following possible cases:
\begin{enumerate}
\item fibres over $N_c\iso H_c$, i.~e., over $3$ different non-collinear points;
\item fibres over $Z\in N_0$ consisting of a simple point and a double point;
\item fibres  over curvilinear triple  points $Z\in N_0$;
\item fibres  over  non-curvilinear triple points $Z\in N_0$;
\item fibres over $N'$.
\end{enumerate}
The corresponding fibres will be referred to as fibres of type $(1)$, $(2)$, $(3)$, $(4)$, and $(5)$ respectively.
\subsubsection{Fibres of type~$(1)$}\label{subsubsection: fibres 1-1-1}
Let $Z\in H_c\iso N_c$ be a non-collinear configuration of $3$ points in $\P_2$. Then, after applying an appropriate coordinate change, we can assume without loss of generality that $Z$ is the union of three points $\pt_0=\point{1, 0, 0}$, $\pt_1=\point{0, 1, 0}$, $\pt_2=\point{0, 0, 1}$, the corresponding Kronecker module is
\[
\Phi=\pmat{x_0&x_1&0\\x_0&0&x_2},
\]
whose minors
$d_0=x_1x_2$, $d_1=-x_0x_2$, $d_2=-x_0x_1$ generate the ideal $I_Z$ of $Z$.
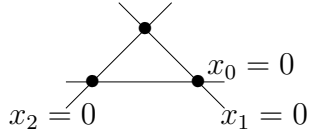
\begin{center}
\begin{tikzpicture}
\node at (20pt, 0) {\textbullet};
\node at (40pt, 6pt) {$x_0=0$};
\node at (45pt, -13pt) {$x_1=0$};
\node at (-35pt, -13pt) {$x_2=0$};
\node at (-20pt, 0) {\textbullet};
\node at (0, 20pt) {\textbullet};
\draw[line width=0mm] (-30pt, 0) to (30pt, 0);
\draw[line width=0mm] (-30pt, -10pt) to ++(40pt, 40pt);
\draw[line width=0mm] (30pt, -10pt) to ++(-40pt, 40pt);
\end{tikzpicture}
\end{center}

The fibre of $\nu$ over the class of  $\Phi$ in $N_0$ consists of the orbits of the matrices
\[
A=\pmat{q_0(x_0,x_1,x_2)&q_1(x_0, x_2)&q_2(x_0, x_1)\\x_0&x_1&0\\x_0&0&x_2}.
\]
The coefficients of
\begin{align}\label{eq: coef q}
\begin{split}
q_0&=a_0x_0^2+a_1x_0x_1+a_2x_0x_2+a_3x_1^2+a_4x_1x_2+a_5x_2^2,\\
q_1&=b_0x_0^2+b_2x_0x_2+b_5x_2^2,\\
q_2&=c_0x_0^2+c_1x_0x_1+c_3x_1^2
\end{split}
\end{align}
can be seen as the projective coordinates of the fibre $\nu^{-1}([\Phi])\iso \P_{11}$.

The ideal that defines the subvariety  corresponding to the singular sheaves is computed by eliminating the variables $x_0, x_1, x_2$ from the saturation of $I_Z$ with respect to the non-essential maximal ideal $(x_0, x_1, x_2)$. We perform the computations using the computer algebra system
\textsc{Singular} (cf.~\cite{SingularProgram}).

We get the ideal (see~\ref{Singular: fibres over 1-1-1} for computations)
\[
(b_0,c_0)\cap(a_3, c_3)\cap (a_5, b_5),
\]
i.~e., the fibre of $M'_0$ over $[\Phi]$ is a union of $3$ components, each being a projective subspace in $\P_{11}$ of codimension $2$. The components lie in a general position: each two components intersect along a projective subspace of codimension $4$ and the intersection of all three of them is a projective subspace of codimension $6$.

\subsubsection{Fibres of type $(2)$}\label{subsubsection: fibres 1-2}
Let $Z\in H_0\setminus H_c$ be a non-collinear configuration of a simple point $\pt_1$ and a double non-collinear point at $\pt_2$. The double point is defined by the underlying simple point $\pt_2$ and a tangent vector at $\pt_2$. Since $Z$ does not lie on a line, the tangent vector should be normal to the line joining $\pt_1$ and $\pt_2$. Therefore, after applying an appropriate coordinate change, we can assume without loss of generality that $\pt_1=\point{0, 0, 1}$, $\pt_2=\point{0, 1, 0}$, and the tangent vector at $\pt_2$ is parallel to the line given by $x_2$.
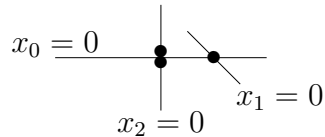
\begin{center}
\begin{tikzpicture}
\draw[line width=0mm] (-40pt, 0) to (40pt, 0);
\draw[line width=0mm] (0, -20pt) to (0, 20pt);
\draw[line width=0mm] (20pt, 0) to ++(10pt, -10pt);
\draw[line width=0mm] (20pt, 0) to ++(-10pt, 10pt);
\node at (0, 2pt) {\textbullet};
\node at (0, -2pt) {\textbullet};
\node at (20pt,0) {\textbullet};
\node at (-40pt,5pt) {$x_0=0$};
\node at (0, -25pt) {$x_2=0$};
\node at (45pt, -15pt) {$x_1=0$};
\end{tikzpicture}
\end{center}
The ideal of $Z$ equals $(x_0, x_1)\cap (x_0^2, x_2)$, the corresponding Kronecker module can be taken to be
\[
\Phi=\pmat{
x_0&x_1&0\\
0&x_0&x_2
}
\]
The fibre of $\nu$ over the class of  $\Phi$ in $N_0$ consists of the orbits of the matrices
\[
A=\pmat{q_0(x_0,x_1,x_2)&q_1(x_0, x_2)&q_2(x_0, x_1)\\
x_0&x_1&0\\
0&x_0&x_2
}.
\]
The coefficients of $q_0$, $q_1$, $q_2$ as in~\refeq{eq: coef q}
can be seen as the projective coordinates of the fibre $\nu^{-1}([\Phi])\iso \P_{11}$.

The fibre of $M'_0$ over $[\Phi]\in N$ is given by the ideal
\[
( a_3^2, c_3^2, a_3c_3, a_1c_3-a_3c_1)\cap (a_5, b_5)
\]
whose radical is
\(
(a_3, c_3)\cap (a_5, b_5),
\)
which means that the fibre consists of two components each of which is a projective subspace of $\nu^{-1}([\Phi])\iso \P_{11}$ of codimension $2$. For computations see~\ref{Singular: fibres over 1-2}.

\subsubsection{Fibres  of type $(3)$}\label{subsubsection: fibres 3 curvilinear}
Let $Z$ be a a triple curvilinear point. Without loss of generality, applying an appropriate coordinate change if necessary, we can assume that $Z$ is supported at $\pt=\point{1,0,0}$ and the ideal of  $Z$  in the affine coordinates $x=x_1/x_0$, $y=x_2/x_0$ is this case
\[
(y^3, x-sy-t^{-1}y^2), \quad s\in \k, \quad t\in \k^*.
\]
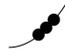
\begin{center}
\begin{tikzpicture}
\node at (0.5, 0.25) {\textbullet};
\node at (0.6, 0.36) {\textbullet};
\node at (0.4, 0.16) {\textbullet};
\draw[domain=0:.8,smooth,variable=\x,line width=0mm] plot ({\x},{\x*\x});
\end{tikzpicture}
\end{center}
The corresponding Kronecker module can be taken to be
\[
\Phi=\pmat{
x_2+2stx_0&x_1-sx_2&tx_0\\
x_1+sx_2&0&x_2
}.
\]
The fibre of $\nu$ over the class of  $\Phi$ in $N_0$ consists of the orbits of the matrices
\begin{equation}\label{eq: 3 curvilinear}
A=\pmat{q_0(x_0, x_1, x_2)&q_1(x_0, x_2)&q_2(x_0, x_1)\\
x_2+2stx_0&x_1-sx_2&tx_0\\
x_1+sx_2&0&x_2
}.
\end{equation}
The coefficients of $q_0$, $q_1$, $q_2$ as in~\refeq{eq: coef q}
can be seen as the projective coordinates of the fibre $\nu^{-1}([\Phi])\iso \P_{11}$.

The fibre of $M'_0$ over $[\Phi]\in N$ is given by the ideal
whose radical is
\[
(b_0, a_0-2sc_0),
\]
which means that the fibre consists of one component which  is a projective subspace of codimension $2$ in $\nu^{-1}([\Phi])\iso \P_{11}$. For computations see~\ref{Singular: fibres over 3}.

\subsubsection{Fibres  of type $(4)$}\label{subsubsection: 3 non-curvilinear}
Let $Z$ be a non-curvilinear triple point. After a change of coordinates we may assume that $Z$ is supported at $\pt=\point{1,0,0}$. Since there is only one non-curvilinear triple point at a given point of a smooth surface, the ideal of $Z$ equals $(x_1^2,x_1x_2, x_2^2)$, the corresponding Kronecker module can be taken to be
\[
\Phi=\pmat{
x_2&x_1&0\\
x_1&0&x_2
}.
\]
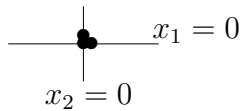
\begin{center}
\begin{tikzpicture}
\draw[line width=0mm] (-1, 0) to (1, 0);
\draw[line width=0mm] (0, 0.5) to (0, -0.5);
\node at (0,0) {\textbullet};
\node at (0,3pt) {\textbullet};
\node at (3pt, 0) {\textbullet};
\node at (1.5, 5pt) {$x_1=0$};
\node at (2pt, -0.7) {$x_2=0$};

\end{tikzpicture}
\end{center}
The fibre of $\nu$ over the class of  $\Phi$ in $N_0$ consists of the orbits of the matrices
\[
A=\pmat{q_0(x_0,x_1,x_2)&q_1(x_0, x_2)&q_2(x_0, x_1)\\
x_2&x_1&0\\
x_1&0&x_2
}.
\]
By Lemma~\ref{lemma: singularMin} all such matrices define singular sheaves since all the $(2\times 2)$-minors vanish at $\pt$. Therefore, $M'_0$ is a $\P_{11}$-bundle  over the locus of non-curvilinear triple points.
\subsubsection{Fibres of type $(5)$}
Let $[\Phi]\in N'$, then without loss of generality
\[
\Phi=\pmat{x_0&0&x_1\\0&x_0&x_2}
\] and the fibre of $\nu$ over $[\Phi]$ consists of the orbits of the matrices~\refeq{eq: fibre over N'}. By Lemma~\ref{lemma: singularMin} the sheaf defined by
\[
\pmat{q_0(x_1, x_2)&q_1(x_1, x_2)&q_2(x_0, x_1, x_2)\\x_0&0&x_1\\0&x_0&x_2}
\]
is singular if and only if the quadratic forms
\[
q_0(x_1, x_2)=a_3x_1^2+a_4x_1x_2+a_5x_2^2 \quad \text{ and }\quad q_1(x_1, x_2)=b_3x_1^2+b_4x_1x_2+b_5x_2^2
\]
have a common zero. The latter holds if and only if the resultant of $q_0$ and $q_1$
\[
R=R(q_0, q_1)(a_3, a_4, a_5, b_3, b_4, b_5)
\]
 vanishes. Since $R$ is an irreducible homogeneous polynomial of degree $4$ in variables $a_3$, $a_4$, $a_5$, $b_3$, $b_4$, $b_5$, the fibres over $N'$ are open subsets of irreducible hyper-surfaces of degree $4$ in $\P_{11}$.  These subsets are obtained by throwing away the points corresponding to matrices with zero determinant, i.~e., the line  $L^*$ (cf.~\ref{subsubsection: fibres over N'}), which is contained in the hypersurface.

\section{Main result}\label{section: main result}
The information about the fibres of $M'_0$ over $N$ obtained in the previous section allows to prove Proposition~\ref{pr: main}.  
\subsection{Dimension}
We showed that the fibres of $M'_0$ over $N$ are generically $9$-dimensional, the fibres are more than $9$-dimensional only over a subvariety of $N$ of dimension $2$.
Therefore, the dimension of $M'_0$ (and thus of $M'$) is $15$, i.~e., $M'$ has codimension $2$ in $M$.

\subsection{Singularities}
Notice that the computation from~\ref{subsubsection: fibres 1-1-1} works also locally over the base. Let us make this clear in the case of $\k=\bb C$, i.~e., in the analytic category with analytic topology.

Let us vary the points $p_0=\point{1, p_{1}^{(0)},p_{2}^{(0)}}$, $p_1=\point{p_{0}^{(1)}, 1,p_{2}^{(1)}}$, $p_2=\point{p_{0}^{(2)},p_{1}^{(2)}, 1}$ in disjoint neighborhoods in $\P_2$ of points $\point{1,0,0}$, $\point{0,1,0}$, $\point{0,0,1}$ respectively. Assume moreover that $p_0$, $p_1$, $p_2$ are always non-collinear.
Then $p_{1}^{(0)}$, $p_{2}^{(0)}$, $p_{0}^{(1)}$, $p_{2}^{(1)}$, $p_{0}^{(2)}$, $p_{1}^{(2)}$ are local coordinates of $N$ around the class of the Kronecker module
\[
\Phi=\pmat{x_0&x_1&0\\x_0&0&x_2}.
\]
Denote by $U_{p_0, p_1, p_2}$ the corresponding neighborhood of $[\Phi]$.

Let $\bar x_i$, $i=0, 1, 2$, be a linear form that defines the line not passing through $p_i$ and passing through the other two points.

The fibre of $\nu$ over the class of
\[
\bar\Phi=\pmat{\bar x_0&\bar x_1&0\\\bar x_0&0&\bar x_2}
\]
consists of the orbits of the matrices
\[
A=\pmat{\bar q_0(\bar x_0,\bar x_1,\bar x_2)&\bar q_1(\bar x_0, \bar x_2)&\bar q_2(\bar x_0, \bar x_1)\\\bar x_0&\bar x_1&0\\\bar x_0&0&\bar x_2}.
\]
The coefficients of
\begin{align*}
\bar q_0&=a_0\bar x_0^2+a_1\bar x_0\bar x_1+a_2\bar x_0\bar x_2+a_3\bar x_1^2+a_4\bar x_1\bar x_2+a_5\bar x_2^2,\\
\bar q_1&=b_0\bar x_0^2+b_2\bar x_0\bar x_2+b_5\bar x_2^2,\\
\bar q_2&=c_0\bar x_0^2+c_1\bar x_0\bar x_1+c_3\bar x_1^2
\end{align*}
can be seen as the projective coordinates of the fibre $\nu^{-1}([\bar \Phi])\iso \P_{11}$, this gives a trivialization of $\bb B$ around  $[\Phi]$. As in~\ref{subsubsection: fibres 1-1-1} we conclude that $M'$ over $U_{p_0, p_2, p_3}$ is a trivial bundle with the fibre computed in~\ref{subsubsection: fibres 1-1-1}. Therefore, $M'_c=M'\cap M_c$ is a bundle over $N_c$ with this singular fibre, which shows that $M'$ is singular. Our argument shows also that the singularities of $M'_0$ over $N_c$ lie in codimension $2$.


\begin{rem}
Notice that in the algebraic category a modification of the argument above would lead to a local triviality of $M'$ over $N_c$ only in \'etale topology. This would not affect however our conclusions.
\end{rem}
\subsection{Connectedness} As shown in~\ref{subsection: fibres of M'}, every fibre of $M'_0$ over $N$ is (path-)connected. Therefore, since $N$ is (path-)connected, $M'_0$ is (path-)connected. Since $M'_1$, which is isomorphic to the universal singular locus of plane quartic curves,  is (path-)connected, it remains to connect  $M'_1$ with $M'_0$.

The latter can be done, for example, as follows. Let $C$ be a quartic curve with a simple double point singularity $p_0\in C$. Fix a line through $p_0$ that is not a component of $C$ and  intersects $C$ at $3$ different points $p_0$, $p_1$, $p_2$.

Consider a degeneration $Z_t=\{p_0, p_1, p(t)\}$ of a configuration of $3$ non-collinear points on $C$ to the configuration $Z_0=\{p_0, p_1, p_2\}$, i.~e., $p(t)\to p_2$, $t\to 0$.
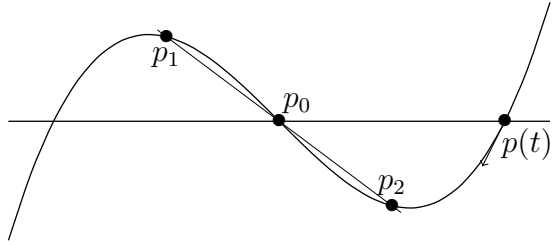
\begin{center}
\begin{tikzpicture}[scale=3]
\node at (0.5, -0.375) {\textbullet};
\node[above] at (0.5, -0.375) {$p_2$};
\node at (-0.5, 0.375) {\textbullet};
\node[below] at (-0.5, 0.375) {$p_1$};
\node at (0, 0) {\textbullet};
\node at (0.08, 0.08) {$p_0$};
\node at (1, 0) {\textbullet};
\node at (1.1, -0.1) {$p(t)$};
\draw[->] (1,0) to  +(-0.1,-0.2);
\draw[line width =0mm ] (0.5, -0.375) to ++(-1.04, 0.78);
\draw[line width =0mm ] (0.5, -0.375) to ++(.04, -.03);
\draw[line width=0.5pt] (-1.2, 0) to (1.2, 0);
\draw[domain=-1.2:1.2,smooth,variable=\x,line width=0.5pt] plot ({\x},{\x*(\x*\x-1)});
\end{tikzpicture}
\end{center}
This gives a degeneration of the twisted ideal sheaf $\cal E_t=\cal I_{Z_t}(1)$ of $Z_t$ in $C$ to the twisted ideal sheaf $\cal I_{Z_0}(1)$ of $Z_0$.
 
 Notice that  $\cal E_0=\cal I_{Z_0}(1)$ is a non-trivial extension~\refeq{eq: ext closed fibre} with $p=p_0$. Therefore, $\cal E_0$ defines a point in $M_1$. Since $p_0$ is a singular point of $C$, as mentioned in~\ref{subsubsection: singular in M_1}, $\cal E$ must be singular sheaf.
On the other hand, as we shall show in Section~\ref{section: sing and sing sheaves},
$\cal E_t$ is a singular sheaf for $t\neq 0$ as $p_0$ is a singular point of $C$.
This gives a path connecting $M'_0$ with $M'_1$.

\section{Singular sheaves and singularities of their support}\label{section: sing and sing sheaves}
Let $[\cal E]\in M_{00}=\bb B_0$. Let $C=\Supp \cal E$ be its support, which is a quartic curve in $\P_2$. As  $\cal E$ is a part of an exact sequence~\refeq{eq: twisted ideal}, it is a subsheaf of $\O_C(1)$, hence a torsion free sheaf on $C$.
Since torsion free sheaves on smooth curves are locally free (see e.g.~\cite[Lemma~5.2.1]{LePotierBook}), we conclude that $\cal E$ is non-singular if $C$ is smooth at all points of $Z$. So $\cal E$ can only be singular if $C$ is singular at some points of $Z$. This demonstrates  that the image of $M'_{00}=M'\cap M_{00}$ under  ~\refeq{eq: map to flags} is included in the subvariety of pairs $(C, Z)$ such that $Z$ contains a singular point of  $C$. We shall demonstrate that $M'_{00}$ generically coincides with this variety.
More precisely, the image of $M_{c}'=M'\cap M_{c}$ under the morphism
\[
 M_{c}\xra{\mu\times \nu} \P S^4V^*\times H_0
\]
consists of the pairs $(C, Z)$, $Z\subset C$, such that $C$ is a singular plane curve of degree $4$ whose singular locus contains at least  one of the points of $Z$.
\begin{pr}\label{pr: singularGeneric}
Let $[\cal E]$ as above belong to $M_c$, then

1) $\cal E$ is singular if and only if $\Sing C\cap Z\neq \emptyset$;

2) the fibre of $M'_0$ over $Z\in H_c$, $Z=\{\pt_0, \pt_1,\pt_2\}$, under the morphism $M'_0\xra{\nu}N_c\iso H_c$ corresponds to the variety of plane quartic curves through $Z$ such that one of the points of $Z$ is a singular point of $C$;

3) for each $i=0,1,2$, the variety of quartics through $Z$ such that $\pt_i$ is a singular point of $C$ coincides with one of three different irreducible components of  the fibre.
\end{pr}
\begin{proof}
Follows from the computations given in~\ref{Singular: fibres over 1-1-1}.
\end{proof}

\begin{rem}
Since $\cal E$ is a twisted ideal sheaf of $3$ different points on a quartic curve (cf.~\refeq{eq: twisted ideal}), the statement 1) of~Proposition~\ref{pr: singularGeneric} immediately follows from Lemma~\ref{lemma:free extension} below.
\end{rem}

\subsection{An observation from commutative algebra}
Let $R=\O_{C, p}$ be a local $\k$-algebra of a curve $C$ at point $p\in C$.
Let $\frak m=\frak m_{C, p}$ be its maximal ideal and let
 $\k_p=R/\frak m$ be the local ring of the structure sheaf of the one point subscheme $\{p\}\subset C$. An $R$-module homomorphism $R\xra{\varphi} \k_p$ is uniquely defined by $\varphi(1)=\lambda\in \k_p$. Then $\varphi(s)=\bar s\cdot \lambda$. If $\varphi$ is different from zero, then the kernel of $\varphi$ coincides with $\frak m$.

\begin{lemma}\label{lemma:free extension}
Consider an  exact sequence of $R$-modules.
\[
0\ra M\ra R\ra \k_p\ra 0
\]
with a non-zero $R$-module $M$.
Then $M$ is free if and only if $R$ is regular.
\end{lemma}
\begin{proof}
If $M$ is free, then  $M\iso R$ (otherwise $M\ra R$ would not be injective)  and we obtain an exact sequence of $R$-modules
\[
0\ra R\ra R\ra \k_p\ra 0,
\]
which
means  that the maximal ideal $\frak m$ of $p$  is generated by one element. Therefore,  $R$ is regular in this case.

Vice versa, assume  $R$ is regular. Notice that $M$ is always a torsion free $R$-module as a submodule of $R$. Therefore, if $R$ is regular, $M$ is free as a torsion free module over a regular one-dimensional local ring.
\end{proof}

\begin{rem}
Notice that Proposition~\ref{pr: singularGeneric} does not hold over $N_0\setminus N_c$. Indeed, take $[\cal E_A]\in M_{00}\setminus M_{c}$ with
\[
A=
\pmat{
x_2^2&0&x_1^2\\
x_0&x_1&0\\
0&x_0&x_2
}.
\]
Then the support $C$ of $\cal E_A$ is given by $x_1(x_2^3+x_0^2x_1)=0$,
one obtains an  exact sequence
\[
0\ra \cal E_A\ra \O_C(1)\ra \O_Z\ra 0
\]
such that $Z$ consists of the simple point $\point{0,0,1}$ and the double point $\point{0, 1, 0}$. In this case $\cal E_A$ is a non-singular sheaf but  $\point{0,1,0}\in Z\cap \Sing C$.

In~\ref{Singular: fibres over 1-2} we compute that every matrix $A$ as in~\ref{subsubsection: fibres 1-2} with $a_3=0$, $a_5\neq 0$ defines a non-singular sheaf, however the intersection of $Z$ with the singular locus of the supporting curve $C$ is non-empty.
\end{rem}

\appendix
\section{Computations of the fibres of $M'_0$ over $N$ with \textsc{Singular} }\label{appendix: computations}
\subsection{Fibres of type $(1)$}\label{Singular: fibres over 1-1-1}
\begin{Verbatim}
                     SINGULAR                                 /  Development
 A Computer Algebra System for Polynomial Computations       /   version 4.0.2
                                                           0<
 by: W. Decker, G.-M. Greuel, G. Pfister, H. Schoenemann     \   Feb 2015
FB Mathematik der Universitaet, D-67653 Kaiserslautern        \
t1.sng   1> LIB "elim.lib"; 
// ** loaded /usr/bin/../share/singular/LIB/elim.lib (4.0.0.1,Jan_2014)
// ** loaded /usr/bin/../share/singular/LIB/ring.lib (4.0.0.0,Jun_2013)
// ** loaded /usr/bin/../share/singular/LIB/primdec.lib (4.0.2.0,Apr_2015)
// ** loaded /usr/bin/../share/singular/LIB/absfact.lib (4.0.0.0,Jun_2013)
// ** loaded /usr/bin/../share/singular/LIB/triang.lib (4.0.0.0,Jun_2013)
// ** loaded /usr/bin/../share/singular/LIB/matrix.lib (4.0.0.0,Jun_2013)
// ** loaded /usr/bin/../share/singular/LIB/nctools.lib (4.0.0.0,Jun_2013)
// ** loaded /usr/bin/../share/singular/LIB/random.lib (4.0.0.0,Jun_2013)
// ** loaded /usr/bin/../share/singular/LIB/poly.lib (4.0.0.0,Jun_2013)
// ** loaded /usr/bin/../share/singular/LIB/general.lib (4.0.0.1,Jan_2014)
// ** loaded /usr/bin/../share/singular/LIB/inout.lib (4.0.0.0,Jun_2013)
t1.sng   2> ring r=0, (x(0..2), a(0..5), b(0..5), c(0..5)), dp;
t1.sng   3> ideal maxm=x(0..2);
t1.sng   4> poly X=x(0)*x(1)*x(2);
t1.sng   5> poly q(0..2);
t1.sng   6> q(0) = a(0)*x(0)^2 + a(1)*x(0)*x(1) + a(2)*x(0)*x(2) + a(3)*x(1)^2 + a(4)*x(1)*x(2) + a(5)*x(2)^2;
t1.sng   7> q(1) = b(0)*x(0)^2 + b(1)*x(0)*x(1) + b(2)*x(0)*x(2) + b(3)*x(1)^2 + b(4)*x(1)*x(2) + b(5)*x(2)^2;
t1.sng   8> q(2) = c(0)*x(0)^2 + c(1)*x(0)*x(1) + c(2)*x(0)*x(2) + c(3)*x(1)^2 + c(4)*x(1)*x(2) + c(5)*x(2)^2;
t1.sng   9> q(1)=subst(q(1), x(1), 0);
t1.sng  10> q(2)=subst(q(2), x(2), 0);
t1.sng  11> // general form of matrices representing the points in the fibre
t1.sng  12. matrix A[3][3] = q(0), q(1), q(2), x(0), x(1), 0, x(0), 0, x(2);
t1.sng  13> print(A);
A[1,1],A[1,2],A[1,3],
x(0),  x(1),  0,     
x(0),  0,     x(2)   
t1.sng  14> // the Kronecker module corresponding to 3 non-collinear points
t1.sng  15. matrix Phi=submat(A, 2..3, 1..3);
t1.sng  16> print(Phi);
x(0),x(1),0,  
x(0),0,   x(2)
t1.sng  17> // the ideal of 2x2 minors of A
t1.sng  18. ideal minm = minor(A, 2);
t1.sng  19> minm = sat(minm, maxm)[1]; // compute its saturation
t1.sng  20> minm = elim(minm, X); // eliminate the variables x(0), x(1), x(2)
t1.sng  21> print(minm);
b(5)*c(0)*c(3),
a(5)*c(0)*c(3),
b(0)*b(5)*c(3),
a(5)*b(0)*c(3),
a(3)*b(5)*c(0),
a(3)*a(5)*c(0),
a(3)*b(0)*b(5),
a(3)*a(5)*b(0)
t1.sng  22> // look at the primary decomposition of the result
t1.sng  23. primdecGTZ(minm);
[1]:
   [1]:
      _[1]=c(3)
      _[2]=a(3)
   [2]:
      _[1]=c(3)
      _[2]=a(3)
[2]:
   [1]:
      _[1]=c(0)
      _[2]=b(0)
   [2]:
      _[1]=c(0)
      _[2]=b(0)
[3]:
   [1]:
      _[1]=b(5)
      _[2]=a(5)
   [2]:
      _[1]=b(5)
      _[2]=a(5)
t1.sng  24> // let us establish a link between the singular sheaves
t1.sng  25. // and the singularities of their support
t1.sng  26. poly f = det(A); // determinant of A
t1.sng  27> // ideal of partial derivatives of f
t1.sng  28. // with respect to x(0), x(1), x(2)
t1.sng  29. // together with the 2x2-minors of Phi,
t1.sng  30. // its zeroes are exactly the singular points of C contained in Z
t1.sng  31. ideal D = diff(f, x(0)), diff(f, x(1)), diff(f, x(2)), minor(Phi, 2);
t1.sng  32> // look at the equations of the subvariety of the fibre defining such sheaves
t1.sng  33. D = sat(D, maxm)[1];
t1.sng  34> D = elim(D, X);
t1.sng  35> // the result coincides with the ideal for singular sheaves
t1.sng  36. primdecGTZ(D);
[1]:
   [1]:
      _[1]=c(3)
      _[2]=a(3)
   [2]:
      _[1]=c(3)
      _[2]=a(3)
[2]:
   [1]:
      _[1]=c(0)
      _[2]=b(0)
   [2]:
      _[1]=c(0)
      _[2]=b(0)
[3]:
   [1]:
      _[1]=b(5)
      _[2]=a(5)
   [2]:
      _[1]=b(5)
      _[2]=a(5)
t1.sng  37> // the ideal of the intersection of pt0 with the singular locus of C
t1.sng  38. D = diff(f, x(0)), diff(f, x(1)), diff(f, x(2)), x(1), x(2);
t1.sng  39> D = sat(D, maxm)[1];
t1.sng  40> D = elim(D, X);
t1.sng  41> // the result coincides with one of the components
t1.sng  42. // of the fibre of singular sheaves computed above
t1.sng  43. primdecGTZ(D);
[1]:
   [1]:
      _[1]=c(0)
      _[2]=b(0)
   [2]:
      _[1]=c(0)
      _[2]=b(0)
t1.sng  44> // the ideal of the intersection of pt0 with the singular locus of C
t1.sng  45. D = diff(f, x(0)), diff(f, x(1)), diff(f, x(2)), x(0), x(2);
t1.sng  46> D = sat(D, maxm)[1];
t1.sng  47> D = elim(D, X);
t1.sng  48> // the result coincides with one of the components
t1.sng  49. // of the fibre of singular sheaves computed above
t1.sng  50. primdecGTZ(D);
[1]:
   [1]:
      _[1]=c(3)
      _[2]=a(3)
   [2]:
      _[1]=c(3)
      _[2]=a(3)
t1.sng  51> // the ideal of the intersection of pt0 with the singular locus of C
t1.sng  52. D = diff(f, x(0)), diff(f, x(1)), diff(f, x(2)), x(0), x(1);
t1.sng  53> D = sat(D, maxm)[1];
t1.sng  54> D = elim(D, X);
t1.sng  55> // the result coincides with one of the components
t1.sng  56. // of the fibre of singular sheaves computed above
t1.sng  57. primdecGTZ(D);
[1]:
   [1]:
      _[1]=b(5)
      _[2]=a(5)
   [2]:
      _[1]=b(5)
      _[2]=a(5)
t1.sng  58> $

$Bye.
\end{Verbatim}
\subsection{Fibres of type $(2)$}\label{Singular: fibres over 1-2}
\begin{Verbatim}
                     SINGULAR                                 /  Development
 A Computer Algebra System for Polynomial Computations       /   version 4.0.2
                                                           0<
 by: W. Decker, G.-M. Greuel, G. Pfister, H. Schoenemann     \   Feb 2015
FB Mathematik der Universitaet, D-67653 Kaiserslautern        \
t2.sng   1> LIB "elim.lib";
// ** loaded /usr/bin/../share/singular/LIB/elim.lib (4.0.0.1,Jan_2014)
// ** loaded /usr/bin/../share/singular/LIB/ring.lib (4.0.0.0,Jun_2013)
// ** loaded /usr/bin/../share/singular/LIB/primdec.lib (4.0.2.0,Apr_2015)
// ** loaded /usr/bin/../share/singular/LIB/absfact.lib (4.0.0.0,Jun_2013)
// ** loaded /usr/bin/../share/singular/LIB/triang.lib (4.0.0.0,Jun_2013)
// ** loaded /usr/bin/../share/singular/LIB/matrix.lib (4.0.0.0,Jun_2013)
// ** loaded /usr/bin/../share/singular/LIB/nctools.lib (4.0.0.0,Jun_2013)
// ** loaded /usr/bin/../share/singular/LIB/random.lib (4.0.0.0,Jun_2013)
// ** loaded /usr/bin/../share/singular/LIB/poly.lib (4.0.0.0,Jun_2013)
// ** loaded /usr/bin/../share/singular/LIB/general.lib (4.0.0.1,Jan_2014)
// ** loaded /usr/bin/../share/singular/LIB/inout.lib (4.0.0.0,Jun_2013)
t2.sng   2> ring r = 0, (x(0..2), a(0..5), b(0..5), c(0..5)), dp;
t2.sng   3> ideal maxm = x(0..2);
t2.sng   4> poly X = x(0)*x(1)*x(2);
t2.sng   5> poly q(0..2);
t2.sng   6> q(0) = a(0)*x(0)^2 + a(1)*x(0)*x(1) + a(2)*x(0)*x(2) + a(3)*x(1)^2 + a(4)*x(1)*x(2) + a(5)*x(2)^2;
t2.sng   7> q(1) = b(0)*x(0)^2 + b(1)*x(0)*x(1) + b(2)*x(0)*x(2) + b(3)*x(1)^2 + b(4)*x(1)*x(2) + b(5)*x(2)^2;
t2.sng   8> q(2) = c(0)*x(0)^2 + c(1)*x(0)*x(1) + c(2)*x(0)*x(2) + c(3)*x(1)^2 + c(4)*x(1)*x(2) + c(5)*x(2)^2;
t2.sng   9> q(1) = subst(q(1), x(1), 0);
t2.sng  10> q(2) = subst(q(2), x(2), 0);
t2.sng  11> // general form of matrices representing the points in the fibre
t2.sng  12. matrix A[3][3] = q(0), q(1), q(2), x(0), x(1), 0, 0, x(0), x(2);
t2.sng  13> print(A);
A[1,1],A[1,2],A[1,3],
x(0),  x(1),  0,     
0,     x(0),  x(2)   
t2.sng  14> // the Kronecker module corresponding to 3 non-collinear points
t2.sng  15. matrix Phi = submat(A, 2..3,1..3);
t2.sng  16> print(Phi);
x(0),x(1),0,  
0,   x(0),x(2)
t2.sng  17> // the ideal of 2x2 minors
t2.sng  18. ideal minm = minor(A, 2);
t2.sng  19> minm = sat(minm, maxm)[1]; // compute its saturation
t2.sng  20> minm = elim(minm, X); // eliminate the variables x(0), x(1), x(2)
t2.sng  21> minm;
minm[1]=b(5)*c(3)^2
minm[2]=a(5)*c(3)^2
minm[3]=a(3)*b(5)*c(3)
minm[4]=a(3)*a(5)*c(3)
minm[5]=a(3)*b(5)*c(1)-a(1)*b(5)*c(3)
minm[6]=a(3)*a(5)*c(1)-a(1)*a(5)*c(3)
minm[7]=a(3)^2*b(5)
minm[8]=a(3)^2*a(5)
t2.sng  22> // look at the primary decomposition of the result
t2.sng  23. primdecGTZ(minm);
[1]:
   [1]:
      _[1]=c(3)^2
      _[2]=a(3)*c(3)
      _[3]=a(3)^2
      _[4]=-a(3)*c(1)+a(1)*c(3)
   [2]:
      _[1]=c(3)
      _[2]=a(3)
[2]:
   [1]:
      _[1]=b(5)
      _[2]=a(5)
   [2]:
      _[1]=b(5)
      _[2]=a(5)
t2.sng  24> // polynomial defining the quartic curve C
t2.sng  25. poly f=det(A);
t2.sng  26> // ideal of singularities of the curve C lying on Z
t2.sng  27. ideal D = diff(f, x(0)), diff(f, x(1)), diff(f, x(2)), minor(Phi, 2);
t2.sng  28> // compute the equations of the subvariety of the corresponding sheaves
t2.sng  29. D = sat(D, maxm)[1];
t2.sng  30> D = elim(D, X);
t2.sng  31> D;
D[1]=a(3)*b(5)*c(3)
D[2]=a(3)*a(5)*c(3)
D[3]=a(3)^2*b(5)
D[4]=a(3)^2*a(5)
t2.sng  32> // look at its primary decomposition
t2.sng  33. // the corresponding variety has an extra component
t2.sng  34. // whose points do not define singular sheaves
t2.sng  35. primdecGTZ(D);
[1]:
   [1]:
      _[1]=a(3)
   [2]:
      _[1]=a(3)
[2]:
   [1]:
      _[1]=b(5)
      _[2]=a(5)
   [2]:
      _[1]=b(5)
      _[2]=a(5)
[3]:
   [1]:
      _[1]=c(3)
      _[2]=a(3)^2
   [2]:
      _[1]=c(3)
      _[2]=a(3)
t2.sng  36> $
$Bye.
\end{Verbatim}
\subsection{Fibres of type $(3)$}\label{Singular: fibres over 3}
\begin{Verbatim}
                     SINGULAR                                 /  Development
 A Computer Algebra System for Polynomial Computations       /   version 4.0.2
                                                           0<
 by: W. Decker, G.-M. Greuel, G. Pfister, H. Schoenemann     \   Feb 2015
FB Mathematik der Universitaet, D-67653 Kaiserslautern        \
t3.sng   1> LIB "elim.lib";
// ** loaded /usr/bin/../share/singular/LIB/elim.lib (4.0.0.1,Jan_2014)
// ** loaded /usr/bin/../share/singular/LIB/ring.lib (4.0.0.0,Jun_2013)
// ** loaded /usr/bin/../share/singular/LIB/primdec.lib (4.0.2.0,Apr_2015)
// ** loaded /usr/bin/../share/singular/LIB/absfact.lib (4.0.0.0,Jun_2013)
// ** loaded /usr/bin/../share/singular/LIB/triang.lib (4.0.0.0,Jun_2013)
// ** loaded /usr/bin/../share/singular/LIB/matrix.lib (4.0.0.0,Jun_2013)
// ** loaded /usr/bin/../share/singular/LIB/nctools.lib (4.0.0.0,Jun_2013)
// ** loaded /usr/bin/../share/singular/LIB/random.lib (4.0.0.0,Jun_2013)
// ** loaded /usr/bin/../share/singular/LIB/poly.lib (4.0.0.0,Jun_2013)
// ** loaded /usr/bin/../share/singular/LIB/general.lib (4.0.0.1,Jan_2014)
// ** loaded /usr/bin/../share/singular/LIB/inout.lib (4.0.0.0,Jun_2013)
t3.sng   2> ring r = (0,s, t), (x(0..2), a(0..5), b(0..5), c(0..5)), dp;
t3.sng   3> ideal I = x(2)^3, x(1)*x(0)-s*x(2)*x(0)-(1/t)*x(2)^2;
t3.sng   4> I=sat(I, x(0))[1];
t3.sng   5> I;
I[1]=x(1)*x(2)+(-s)*x(2)^2
I[2]=x(1)^2+(-2*s)*x(1)*x(2)+(s^2)*x(2)^2
I[3]=(t)*x(0)*x(1)+(-s*t)*x(0)*x(2)-x(2)^2
I[4]=x(2)^3
t3.sng   6> ideal J = I[1..3];
t3.sng   7> std(J);
_[1]=x(1)*x(2)+(-s)*x(2)^2
_[2]=x(1)^2+(-2*s)*x(1)*x(2)+(s^2)*x(2)^2
_[3]=(t)*x(0)*x(1)+(-s*t)*x(0)*x(2)-x(2)^2
_[4]=x(2)^3
t3.sng   8> // thus I = J
t3.sng   9. I = J;
t3.sng  10> matrix S[2][3] = (2*s*t)*x(0)+x(2), x(1)-(s)*x(2), (t)*x(0), x(1)+(s)*x(2), 0 ,x(2);
t3.sng  11> print(S);
(2*s*t)*x(0)+x(2),x(1)+(-s)*x(2),(t)*x(0),
x(1)+(s)*x(2),    0,             x(2)     
t3.sng  12> // the ideal of maximal minors coincides with I
t3.sng  13. minor(S,2);
_[1]=x(1)*x(2)+(-s)*x(2)^2
_[2]=(-t)*x(0)*x(1)+(s*t)*x(0)*x(2)+x(2)^2
_[3]=-x(1)^2+(s^2)*x(2)^2
t3.sng  14> ideal maxm=x(0..2);
t3.sng  15> poly X=x(0)*x(1)*x(2);
t3.sng  16> poly q(0..2);
t3.sng  17> q(0) = a(0)*x(0)^2 + a(1)*x(0)*x(1) + a(2)*x(0)*x(2) + a(3)*x(1)^2 + a(4)*x(1)*x(2) + a(5)*x(2)^2;
t3.sng  18> q(1) = b(0)*x(0)^2 + b(1)*x(0)*x(1) + b(2)*x(0)*x(2) + b(3)*x(1)^2 + b(4)*x(1)*x(2) + b(5)*x(2)^2;
t3.sng  19> q(2) = c(0)*x(0)^2 + c(1)*x(0)*x(1) + c(2)*x(0)*x(2) + c(3)*x(1)^2 + c(4)*x(1)*x(2) + c(5)*x(2)^2;
t3.sng  20> q(1) = subst(q(1), x(1), 0);
t3.sng  21> q(2) = subst(q(2), x(2), 0);
t3.sng  22> // the linear part is S
t3.sng  23. matrix A[3][3];
t3.sng  24> A = q(0), q(1), q(2), (2*s*t)*x(0)+x(2), x(1)-(s)*x(2), (t)*x(0), x(1)+(s)*x(2), 0, x(2);
t3.sng  25> print(A);
A[1,1],           A[1,2],        A[1,3],  
(2*s*t)*x(0)+x(2),x(1)+(-s)*x(2),(t)*x(0),
x(1)+(s)*x(2),    0,             x(2)     
t3.sng  26> // the linear part Phi
t3.sng  27. matrix Phi = submat(A, 2..3, 1..3);
t3.sng  28> //ideal of 2x2 minor sof A
t3.sng  29. ideal minm = minor(A, 2);
t3.sng  30> // compute the ideal of the subvariety of singular sheaves in the fibre
t3.sng  31. minm = sat(minm, maxm)[1];
t3.sng  32> minm = elim(minm, X);
t3.sng  33> // look at its primary decomposition
t3.sng  34. list PD = primdecGTZ(minm);
t3.sng  35> // it has only one component
t3.sng  36. size(primdecGTZ(minm));
1
t3.sng  37> // the corresponding prime ideal is
t3.sng  38. PD[1][2];
_[1]=b(0)
_[2]=a(0)+(-2*s)*c(0)
t3.sng  39> $

$Bye.
\end{Verbatim}

\def\cprime{$'$} \def\cprime{$'$} \def\cprime{$'$}

\end{document}